\documentclass[english,11pt,a4paper,reqno]{amsart}
\usepackage{amsmath,amsthm,amsfonts,amssymb}

\usepackage{a4wide}
\usepackage[ansinew]{inputenc}
\usepackage{graphicx}

\newcommand{\leb}{\operatorname{Leb}}

\newcommand{\dist}{\operatorname{dist}}
\newcommand{\diam}{\operatorname{diam}}
\newcommand{\supp}{\operatorname{supp}}

\begin{document}

\newcommand{\mcup}{\mbox{$\bigcup$}}
\newcommand{\mcap}{\mbox{$\bigcap$}}

\def \RR {{\mathbb R}}
\def \ZZ {{\mathbb Z}}
\def \NN {{\mathbb N}}
\def \PP {{\mathbb P}}
\def \TT {{\mathbb T}}
\def \II {{\mathbb I}}
\def \JJ {{\mathbb J}}

\def \vare {\varepsilon }

 \def \cf {\mathcal{F}}
 \def \cm {\mathcal{M}}
 \def \cn {\mathcal{N}}
 \def \cq {\mathcal{Q}}
 \def \cp {\mathcal{P}}
 \def \cb {\mathcal{B}}
 \def \cc {\mathcal{C}}
 \def \cs {\mathcal{S}}
 \def \bc {\mathcal{B}}
 \def \hc {\mathcal{C}}

\newcommand{\dem}{\begin{proof}}
\newcommand{\cqd}{\end{proof}}

\newcommand{\qand}{\quad\text{and}\quad}

\newtheorem{theorem}{Theorem}
\newtheorem{corollary}{Corollary}

\newtheorem*{Maintheorem}{Main Theorem}
\newtheorem*{Theorem*}{Theorem}

\newtheorem{maintheorem}{Theorem}
\renewcommand{\themaintheorem}{\Alph{maintheorem}}
\newcommand{\cmt}{\begin{maintheorem}}
\newcommand{\fmt}{\end{maintheorem}}

\newtheorem{maincorollary}[maintheorem]{Corollary}
\renewcommand{\themaintheorem}{\Alph{maintheorem}}
\newcommand{\cmc}{\begin{maincorollary}}
\newcommand{\fmc}{\end{maincorollary}}

\newtheorem{T}{Theorem}[section]
\newcommand{\ct}{\begin{T}}
\newcommand{\ft}{\end{T}}

\newtheorem{Corollary}[T]{Corollary}
\newcommand{\cco}{\begin{Corollary}}
\newcommand{\fco}{\end{Corollary}}

\newtheorem{Proposition}[T]{Proposition}
\newcommand{\cpr}{\begin{Proposition}}
\newcommand{\fpr}{\end{Proposition}}

\newtheorem{Lemma}[T]{Lemma}
\newcommand{\cle}{\begin{Lemma}}
\newcommand{\fle}{\end{Lemma}}

\newtheorem{Sublemma}[T]{Sublemma}
\newcommand{\csle}{\begin{Sublemma}}
\newcommand{\fsle}{\end{Sublemma}}

\newtheorem*{Conjecture}{Conjecture}

\theoremstyle{remark}

\newtheorem{Remark}[T]{Remark}
\newcommand{\cre}{\begin{Remark}}
\newcommand{\fre}{\end{Remark}}

\newtheorem{Definition}[T]{Definition}
\newcommand{\cd}{\begin{Definition}}
\newcommand{\fd}{\end{Definition}}

\newtheorem{Question}{Question}


\title[Geometry of expanding
absolutely continuous invariant measures]{Geometry of expanding
absolutely continuous invariant measures
and the liftability  problem}

\author{José F. Alves}
\address{Jos\'e F. Alves\\ Departamento de Matem\'atica, Faculdade de Ci\^encias da Universidade do Porto\\
Rua do Campo Alegre 687, 4169-007 Porto, Portugal}
\email{jfalves@fc.up.pt} \urladdr{http://www.fc.up.pt/cmup/jfalves}

\author{Carla L. Dias}
\address{Carla L. Dias\\ Instituto Politécnico de Portalegre, Lugar da Abadessa,
Apartado 148,
7301-901 Portalegre, Portugal}
\email{carlald.dias@gmail.com}

\author{Stefano Luzzatto}
\address{Stefano Luzzatto\\ Mathematics Department, Imperial College\\
180 Queen's Gate, London SW7, UK}
\email{luzzatto@ictp.it}
\urladdr{http://www.ictp.it/$\sim$luzzatto}
\curraddr{Abdus Salam International Centre for Theoretical Physics, Strada Costiera 11, Trieste. Italy. }
\date{\today}

\thanks{We would like to thank Vilton Pinheiro for sharing his ideas and for many inspiring discussions, and Imre Toth for reading a preliminary version of the paper and making several useful comments.
Work carried out mainly at Imperial College London and the University of Porto. JFA was partially supported by
FCT through CMUP and by POCI/MAT/61237/2004. CLD was
supported by FCT}



\begin{abstract}
We show that for a large class of maps on manifolds of arbitrary finite dimension, the existence of a Gibbs-Markov-Young structure (with Lebesgue as the reference measure)  is a necessary as well as sufficient condition for the existence of an  invariant probability measure which is absolutely continuous measure (with respect to Lebesgue)  and for which all Lyapunov exponents are positive. 
 \end{abstract}

\maketitle

\section{Introduction and statement of results}

\subsection{Background and main definitions}
In the 1960's, Sinai and Bowen showed that all \emph{smooth uniformly hyperbolic} dynamical systems admit a  \emph{finite Markov partition} \cite{Bow70, Sin68}. Sinai, Ruelle and Bowen then used this remarkable geometric structure, and the associated symbolic coding of the system, to study the ergodic properties such as the rate of decay of correlations. Attempts to extend this approach to systems with discontinuities and/or  satisfying weaker non-uniform hyperbolicity conditions by constructing countable Markov partitions has had some, but limited, success, see \cite{KruTro92}, in part due to the difficulty of constructing such partitions and in part due to the difficulty of understanding the ergodic theory of countable subshifts, though there have also been significant advances recently on this latter area \cite{AarDen02, Bre99, BoyBuzGom06, Gur69, Sar99, Sar03}. 

About ten years ago, L.-S. Young  proposed an alternative geometric structure, which we shall call a \emph{Gibbs-Markov-Young (GMY) structure}, as a way of  studying the ergodic properties of certain dynamical systems \cite{Y1, Y2}. In her pioneering papers, Young showed that a GMY structure contains information about several ergodic properties of the system such as for example the rate of decay of correlations. She also showed that classical results for uniformly hyperbolic systems could be recovered in this framework (in fact it is straightforward to show that any system with a finite Markov partition also admits a GMY structure) and that GMY structures exist in more general situation where the classical approach fails. Over the last ten years, this approach has proved to be one of the most successful strategies for understanding the ergodic properties of large classes of systems, with some papers focussing on  the consequences of having a GMY structure, e.g.  \cite{BuzMau05, Gou04, Mel09, MelNic05}, and others focussing on the construction of such structures, e.g.  \cite{Al,  BruLuzStr03, Che99, CheZha05, Fre05,Y2}. 
Notwithstanding these results, we still do not have a complete characterization of systems which admit a GMY structure. A natural generalization of the results of Sinai and Bowen to the smooth non-uniformly hyperbolic setting would be the following
\begin{Conjecture}
A dynamical system admits a GMY structure iff it is non-uniformly hyperbolic.
\end{Conjecture}

We remark that in general one thinks of  non-uniformly hyperbolicity as having non-zero Lyapunov exponents with respect to some invariant probability measure. Moreover the notion of GMY structure presupposes a  (not necessarily invariant)  reference measure. In this paper we will focus on the reference measure as being Lebesgue and the invariant probability measure being absolutely continuous with respect to Lebesgue. 
In this setting,  
the main purpose of this paper is to prove this conjecture in the endomorphism case in which all directions are (non-uniformly) expanding, i.e. when all the Lyapunov exponents are positive.  We shall concentrate first of all on the case of \( C^{2} \) endomorphisms (with and without critical points), but our techniques also give us an almost  complete characterization result in the expanding case admitting singularities with unbounded derivative and also a slightly different  almost complete characterization in the case in which the map has both critical points and singularities.

To give the precise definitions suppose that
 \( M \) is a compact Riemannian manifold of dimension \( d\geq 1 \), \( \leb \) (Lebesgue measure) is the normalized Riemannian volume on \( M \) and
  \( f: M \to M \) is a measurable map which is differentiable almost everywhere (we shall be more specific about the regularity assumptions in the statement of results below).

\begin{Definition}\label{ind}
 Given a ball \( \Delta \subseteq M \), we say that
 \( F: \Delta \to \Delta \) is an \emph{induced map} if
 \( F(x) = f^{R(x)}(x) \) and
 \(R: \Delta \to \mathbb N \)
  is an  \emph{inducing time} function with the property that
  \( f^{R(x)}(x)\in\Delta \) whenever \( x\in \Delta \).
  We say that an induced map \( F: \Delta \to \Delta \)
  is \emph{GMY}
 if 
 there exists a ($\leb$ mod 0) partition \( \mathcal P \) of \( \Delta \) into
 open subsets
 such that \( R \) is constant on each element \( U\in\mathcal P \)
 and $F|_{U}$
is a uniformly expanding diffeomorphism  onto $\Delta$
with
uniformly bounded volume distortion: more precisely,  there are $0<\kappa<1$
and $K>0$ such that for all $U\in\cp$ and all $x,y\in U$
\begin{enumerate}
  \item[\emph{i)}] $\|DF(x)^{-1}\|<\kappa$;
  \item[\emph{ii)}] $\displaystyle\log\left|\frac{\det DF(x)}{\det DF(y)}\right|\leq K \dist(F(x),
F(y)).$
\end{enumerate}
Moreover, if the inducing time function $R$ is integrable with respect to $\leb$, 
then we say that the induced map has integrable return times.
We say that
\( f\)\emph{ admits a GMY structure} if it admits a GMY induced map with integrable return times.
\end{Definition}

We remark that as we are considering Lebesgue as the reference measure, this definition only includes a special case of the more general definition given by Young in \cite{Y2}. In particular if not all directions are expanding or if the reference measure is not Lebesgue, this definition has to be generalized (it may be necessary for example, to induce on a Cantor set, see \cite{Y1}), but what we give here is sufficient for our purposes.

 \begin{Definition}
We say that an invariant probability measure \( \mu \) is \emph{expanding} if all its Lyapunov exponents are positive, i.e. for \( \mu \)-almost every \( x \) and every \( v\in T_{x}M\setminus \{0\} \),
\begin{equation}\label{exp}
\lambda(x,v):=\limsup_{n\to\infty}\frac 1n \log \|Df^{n}(x) v\| > 0.
\end{equation}
\end{Definition}

\subsection{Geometry of expanding measures}
We can now state our result in the simplest but already non-trivial case. Here and in the rest of the paper we shall  use the standard abbreviation of the term ``absolutely continuous (with respect to Lebesgue) invariant probability'' to \emph{acip}.

\begin{theorem}\label{t.locdif}
Let \( f: M \to M \) be a \( C^{2} \) local diffeomorphism. Then \( f \) admits a GMY structure if and only if it admits an ergodic  expanding acip. \end{theorem}

\( C^{2} \) local diffeomorphisms can be non-uniformly expanding (and strictly \emph{not} uniformly expanding) in non-trivial ways. Our result applies in particular to the class of examples of expanding local diffeomorphisms constructed in \cite{ABV}. Nevertheless, many interesting examples are of course not local diffeomorphisms but have critical points (for example one-dimensional Collet-Eckmann maps or higher dimensional Viana maps). Our results apply to such cases under some very mild nondegeneracy conditions on the set of critical points. These are analogous to the notion of a ``non-flat critical point'' in the one dimensional setting, which is essentially a point where at least some higher order derivative does not vanish.

\begin{Definition}
We say that \( x\) is a \emph{critical point} if  \( Df(x) \) is not
invertible. We denote the set of critical points by \( \mathcal C \) and, for every \( n\geq 0 \), let \( \mathcal C_{n}=\cup_{i=0}^{n}f^{-n}(\mathcal C) \) and let \(
\dist(x, \mathcal C_{n}) \) denote the distance between the point \( x \)
and the set \( \mathcal C_{n} \).
We say that a critical set \( \mathcal C \) is
\emph{non-degenerate} if   for every \( n\geq 0 \)
 there are constants $B>1$ and $\beta, \beta' >0$ (possibly depending on \( n \))
 such that for every $x\in
 M\setminus\mathcal C_{n} $
\begin{enumerate}
\item[(C1)]
\quad $  B^{-1}\dist(x,\mathcal C_{n})^{\beta}\leq
 \|Df(x)^{-1}\|^{-1} \leq B\dist(x,\mathcal C_{n})^{\beta'}$.
\end{enumerate}
Moreover, the functions \(  \log|\det Df^{n}| \) and \( \log \|(Df^{n})^{-1}\| \)
are \emph{locally Lipschitz} at points \( x\in M \setminus \mathcal
C_{n} \): for every $x,y\in M\setminus \mathcal C_{n}$ with
$\dist(x,y)<\dist(x,\mathcal C_{n} )/2$ we have
\begin{enumerate}
\item[(C2)] \quad $\displaystyle{\left|\log\|Df^{n}(x)^{-1}\|-
\log\|Df^{n}(y)^{-1}\|\:\right|\leq B{\dist(x,y)}/{\dist(x,\mathcal C_{n}
)^{\beta}}}$;
 \item[(C3)]
\quad $\displaystyle{\left|\log|\det Df^{n}(x)- \log|\det
Df^{n}(y)|\:\right|\leq B {\dist(x,y)}/{\dist(x,\mathcal C_{n} )^{\beta}}}$.
 \end{enumerate}
 \end{Definition}

Notice that \( \|Df(x)^{-1}\|^{-1} 
\) is the minimum expansion of \( Df \) in any direction.

\begin{theorem}\label{t.crit}
Let \( f: M \to M \) be a \( C^{2} \)
map with a non-degenerate critical set.
Then \( f \) admits a GMY structure if and only if it admits an ergodic  expanding acip.
\end{theorem}

Critical points are not the only way that maps can fail to be local diffeomorphisms. Many interesting and relevant examples, also for applications, arise naturally with discontinuities and/or singularities
(points near which the derivative is unbounded).

\begin{Definition}\label{criticalset}
We say that \( x\) is a \emph{singular point} if  \( Df(x) \)
does not exist, including the case in which \( f \) is discontinuous at \( x \).
 We say that a set of singular
points \( \mathcal C \) is \emph{non-degenerate} if \( \leb(\mathcal
C)=0 \) and, for every \( n\geq 0 \), there are constants $B>1$ and $\beta, \beta' >0$
(possibly depending on \( n \)) such that for every $x\in
 M\setminus\mathcal C_{n} $, conditions~(C2), (C3)  of the previous
 definition are satisfied, and condition~(C1) is replaced by
\begin{enumerate}
\item[(C1')]
\quad $ \displaystyle B^{-1}\dist(x,\mathcal C_{n} )^{-\beta'}\leq
\|Df^{n}(x)\| \leq B\dist(x,\mathcal C_{n} )^{-\beta}$.
\end{enumerate}
 \end{Definition}

For maps which have a non-degenerate singular set we get an almost complete characterization, the only gap occurring due to the fact that a GMY structure does not necessarily imply the following  integrability condition. 

\begin{Definition}
We say that \( \mu \) is \emph{regularly expanding} if it is expanding and in addition we have
\begin{equation}\label{regexp}
 \log\|Df^{-1}\|\in L^{1}(\mu).
 \end{equation}
 \end{Definition}

 Notice that condition \eqref{regexp} implies in particular that the limsup in \eqref{exp} is actually a limit. 
We remark also that the integrability condition \eqref{regexp} is always satisfied in the setting of Theorems~\ref{t.locdif} and \ref{t.crit}. This is immediate in the 
  local diffeomorphism case  since \( \|Df\|  \) and \( \|Df^{-1}\| \) are uniformly bounded above and below, and non-trivial  in the case of \( C^{2} \) maps where it is proved in \cite[Lemma 4.2]{P}, based on \cite[Remark 1.2]{Liu98}.  
In the $C^2$ setting \eqref{regexp} implies also the integrability of \( \log \|Df\| \) since \( \|Df\|  \) is bounded above.

\begin{theorem}\label{t.sing1}
Let \( f: M \to M \) be a \( C^{2} \) map outside a non-degenerate singular set.
   If \( f \) admits a GMY structure then it admits an ergodic expanding acip. Conversely, if \( f \) admits an ergodic regularly expanding acip then it admits a GMY structure.
\end{theorem}

Finally, there are also many systems of interest which have a combination of critical points and singularities and possibly discontinuities near which the derivative is bounded away from zero and infinity (notice that this last situation is not included in the definitions of critical and singular points given above). These cases are actually quite subtle and the  interaction between the critical and singular points can give rise to some significant technical issues.

\begin{Definition}
We say that \( f \) admits a non-degenerate \emph{critical/singular} set \( \mathcal C \) if it is a \( C^{2} \) local diffeomorphism outside a set \( \mathcal C \) on which
  \( Df(x) \) is not
invertible or  does not exist (including the case in which \(
f \) is discontinuous at \( x \)) such that \( \leb(\mathcal
C)=0 \) and, for every \( n\geq 0 \), there are constants $B>1$ and $\beta >0$
 such that for every $x\in
 M\setminus\mathcal C_{n} $, conditions~(C2), (C3)  above
 are satisfied, and condition (C1) or (C1') is replaced by
\begin{enumerate}
\item[(C1'')]
\quad $ \displaystyle B^{-1}\dist(x,\mathcal C_{n})^{\beta}\leq
\|Df^{n}(x)^{-1}\|^{-1}\leq \|Df^{n}(x)\| \leq B\dist(x,\mathcal C_{n} )^{-\beta}$.\end{enumerate}
\end{Definition}

\begin{theorem}\label{t.sing}
    Let \( f: M \to M \) be  a \( C^{2} \)
    local diffeomorphism outside
     a non-degenerate critical/singular set \( \mathcal
    C \). If \( f \) admits a GMY structure then it admits an ergodic expanding acip. Conversely,  if \( f \) admits an ergodic regularly expanding acip \( \mu  \) satisfying \( \log d(x, \mathcal C_{n}) \in L^{1}(\mu) \) then it admits a  GMY structure.
    \end{theorem}

Due to the very weak conditions on the critical set, we allow here both critical points and singularities and even allow the same point to be critical in one direction and singular in the other,  we need to assume the integrability of the logarithm of the distance function to the critical set.

\subsection{The liftability problem}
Our results can also be viewed in the context of the so-called ``liftability problem''. It is a classical result that
a  GMY map
\( F \) admits an ergodic  absolutely continuous invariant
probability
 measure (\emph{acip}) \( \nu \)
 with bounded density; see e.g. \cite[Lemma~
 2]{Y1} but the result goes back, at least in its idea, to the 50's and is often considered  a Folklore Theorem.
 It is then  possible to define a measure
 \begin{equation}\label{eq:lift}
{\mu}= \frac{\sum_{j=0}^{\infty}f^j_{*}(\nu|\{  R > j\})}{\sum_{j=0}^{\infty}\nu(\{  R > j\})},
  \end{equation}
  which is seen by standard arguments to be \( f \)-invariant and absolutely continuous.
The integrability condition with respect to the Riemannian volume
and the bounded density of \( \nu \) imply the integrability of \( R
\) with respect to \( \nu \) and thus guarantees that the
denominator is finite. It follows that  \( \mu \) is an  acip for \( f \).

\begin{Definition} If \( F \) is an induced map of \( f \) and \( \nu \) and \( \mu \) are \( F \)-invariant and \( f \)-invariant probability measures respectively, related by the formula \eqref{eq:lift}, then we say that \( \mu \) is the \emph{projection} of \( \nu \) or that
 \( \nu \) is the \emph{lift} of \(
    \mu \) to the induced map.
 \end{Definition}

A natural question is which measures can be obtained in this way, i.e. which measures admit a lift to a GMY induced map.
A few papers have addressed the issue from various points of view, see for example  \cite{BruTod07, Kel89, PesSenZha08, PesZha07, Zwe05} and in particular \cite{Dob} in which a one-dimensional version of some of the results presented here are obtained, and  \cite {Pin}  in which the liftability problem is studied in great generality (not just absolutely continuous measures) under some sets of assumptions different from, but related to, the assumptions of this paper.
One direction of the implications stated in each of the theorems above can be viewed and formulated in this light. Thus, for example we have
\begin{theorem}\label{th:lift}
Let \( f: M \to M  \) be a \( C^{2} \) map with a non-degenerate critical set. Then
every ergodic  expanding acip is liftable.
\end{theorem}

To see how  Theorem \ref{th:lift} follows from Theorem \ref{t.crit} (and the corresponding liftability statements in the other settings follow from the corresponding theorems) recall first of all that, under the assumption of the existence of an expanding acip \(  \mu  \), Theorem \ref{t.crit} implies the existence of a GMY induced map. Let \( \nu \) be the absolutely continuous ergodic \( F \)-invariant
measure for the GMY map with
integrable return time function~\( R \). 
Thus 
we just need to
discuss the relationship between the original measure \( \mu \) and
the lift of the measure \( \nu \) to the induced GMY map.
From \cite[Lemma~
 2]{Y1}
it follows that $\nu$ has density with respect to Lebesgue measure
on $\Delta$ bounded from above and below by positive constants. Then
 we easily get that $R$ is also Lebesgue integrable. Keeping in mind that this
return time is defined in terms of \( f^{N} \) we define \( \tilde R
= N R \) and the corresponding \( f \)-invariant probability measure \(
\tilde\mu \) by \eqref{eq:lift}.
It just remains to show that \( \tilde\mu = \mu \). This follows
from the standard fact that  we have that \( \tilde\mu \) and \( \mu
\) are both ergodic absolutely continuous \( f \)-invariant measures
which contain \(
 \Delta \) in their support and therefore they must be equal.

\subsection{GMY structure implies expanding acip}
One direction of the implications mentioned in our results is relatively straightforward, namely the fact that a GMY structures imply the existence of an expanding acip. We have already mentioned in the previous paragraph the classical arguments which show that a GMY structure implies the existence of an acip \( \mu \), and it thus only remains to show that \( \mu \) has all positive Lyapunov exponents. Let \(R_n=R_n(x)\) denotes the number of iterations of \( f \) required for \( x \) to have \( n \) returns under the induced map \( F \). Then we can write
\[ \frac 1n \log \| DF^{n}(x) v \|= \frac 1n \log \|Df^{R_{n}(x)}(x)v\|= \frac{R_{n}}{n} \frac{1}{R_{n}}\log\|Df^{R_{n}}(x) v \|. \]
By the integrability of the return times we have \( R_{n}/n  \) converging to some positive constant, and therefore by the positivity of the Lyapunov exponent for \( F \) it follows that the above equation is positive as \( n\to \infty \) and this implies that \( f \) also has positive Lyapunov exponent.

\subsection{Technical remarks and overview of the paper}
In the previous paragraph we have already discussed one direction of the implications in our main Theorems.
Thus, our full attention in the body of the paper is devoted to the construction of a GMY structure relying only on the assumption of the existence of an expanding acip.
This construction consists  of two main steps: the construction of the GMY induced map and the control of the return times in order to ensure integrability. In both of these steps we achieve, over and above the novelty of the results,  major simplification and greater conceptual clarity in comparison to most existing approaches for similar constructions in other settings.
Indeed, in most papers in which a GMY structure is obtained e.g. \cite{ALP, BruLuzStr03, Dia06, DiaHolLuz06, Gou06, Hol05, Y2}
the construction  is quite involved and technical,
using a mixture  of combinatorial, analytic and probabilistic
arguments. Also, in these papers significantly stronger assumptions are used which imply relatively fast rates of decay  (e.g. exponential or
polynomial) of the inducing time function,
depending on various additional assumptions on the map.
In our case, we are here able to implement what is essentially the most naive strategy in order to achieve our goal, namely to choose some small ball \( \Delta \) and iterate it until some subset of \(\Delta \) covers \( \Delta \) in the right way. This subset then becomes one of the elements of the final partition and we repeat the procedure with the remaining points. A crucial tool used here to ensure that that all regions of \( \Delta \) eventually grow sufficiently large is the notion
of \emph{hyperbolic time}.  This idea which was first applied in the setting of non-uniformly expanding maps  in
\cite{Alv00} and has since then been widely applied in a variety of
settings including the construction of induced GMY maps in
some situations such as those considered in \cite{ALP, Gou06} but not in many other constructions such as   \cite{Y2, BruLuzStr03, Hol05, Dia06,
DiaHolLuz06}.
A major benefit of our approach is that it gives  a particularly
 \emph{efficient algorithm}. This allows us to obtain the integrability of the return times with no particular assumptions.
 Similar arguments and related results have been obtained by Pinheiro in a previous paper \cite{P} and a more general recent preprint  \cite{Pin}.

A far as the organization of the paper is concerned,
notice that the statement of Theorem~\ref{t.sing} concerning sufficient conditions for the existence of a GMY map,
includes Theorems~\ref{t.locdif},~\ref{t.crit} and 3 as special cases (we shall show that the additional integrability  condition \( \log d(x, \mathcal C_n)\in L^{1}(\mu) \) assumed explicitly in Theorem 4 is automatically satisfied in the other cases, see Remark \ref{rem.int}).
We shall therefore concentrate on the proof of
the most general setting as formulated in Theorem~\ref{t.sing}.
In Section \ref{nue} we show that some power of \( f \) satisfies
some stronger expansion condition and also some slow recurrence to
the singular set.  These are the standard conditions which are
usually assumed in the setting of so-called \emph{non-uniformly
expanding maps}. In Section~\ref{sub.hyptimes} we recall some known
properties of non-uniformly expanding maps including the crucial
notion of \emph{hyperbolic time}.  We also prove the important fact
that the support of an invariant measure for a non-uniformly
expanding map contains a ball. This is important in our setting
because, unlike the situation in other papers such as \cite{ALP,
Gou06, P}, we are not assuming that the map is non-uniformly
expanding on the whole manifold.
In Section \ref{s.markov} we give the complete construction of the
induced GMY map. 

We mention here some key differences
between our construction and that of \cite{ALP, Gou06, P}. One of
the shortcomings of \cite{ALP} was a relatively inefficient
construction which led to significantly larger inducing times than
necessary, thus allowing only polynomial estimates to be obtained.
This aspect of the construction was improved in \cite{Gou06, P}
where a global partition of the manifold was introduced, leading to
significantly more efficient construction where the inducing times
are essentially optimal. This strategy cannot be used here since
our assumptions do not necessarily imply the map to be non-uniformly
expanding on the whole attractor. We therefore return to a more local
construction but develop a new strategy to improve the effectiveness
of the inducing time estimates. Finally, in Subsection
\ref{s.inducing} we prove the integrability of the inducing times
for the constructed GMY map.

\section{Non-uniform expansion and slow recurrence}
\label{nue}

In this section we prove that there exists some subset \( A \subseteq M \) on which some power of \( f \) satisfies some quite strong expansivity and recurrence conditions.

\cd Let $f:M\to M$ be a $C^2$ local diffeomorphism outside a
non-degenerate critical/singular set $\mathcal{C}$. We say that $f$ is
\emph{non-uniformly expanding} (\emph{NUE}) on a set $A \subset M$
if there is $\lambda>0$ such that for  every  $x\in A$ one has
\begin{equation*} \liminf_{n\to+\infty}
\frac{1}{n}
    \sum_{j=1}^{n} \log \|Df(f^j(x))^{-1}\|<-\lambda.
\end{equation*}
We say that $f$ has \emph{slow recurrence} (\emph{SR}) if given any
$\epsilon>0$ there exists $\delta>0$ such that for every  $x \in A$
we have
\begin{equation*}
\limsup_{n\to+\infty} \frac{1}{n} \sum_{j=1}^{n}-
\log\dist_\delta(f^j(x),\mathcal{C})\leq\epsilon,
\end{equation*}
where
\( \dist_\delta(x,\mathcal{C})= 1
\) if \( \dist(x,\mathcal{C})\geq\delta \) and
\( \dist(x,\mathcal{C}) \)  otherwise.

\fd

The main result of this section is that there exists a set \( A \) on which
some power of \( f \) satisfies the two conditions (NUE) and (SR).

\cpr\label{pr.NUE} Let \( \mu \)
    be an ergodic regularly expanding acip.
Then, for all $N$  large enough, \( f^N \) satisfies NUE and SR on  a forward
$f^N$-invariant set \( A \) with  a positive Lebesgue measure subset of points whose  $f^N$-orbit is dense in $A$.
\fpr

Proposition \ref{pr.NUE} allows us to reduce the proof of our main theorems to the proof of the following

 \begin{theorem}
 \label{t.NUE} Let \( f: M \to M \) be a \(
C^{2} \)
   local diffeomorphism outside
    a non-degenerate critical/singular set \( \mathcal
   C \). Assume that \( f \) satisfies NUE and SR on a forward
invariant set \( A \) with a positive Lebesgue measure subset of points whose orbit is dense in~$A$.
  Then \( f \) admits a GMY structure.
\end{theorem}

Theorem~\ref{t.NUE} and Proposition~\ref{pr.NUE}
 imply a GMY structure \( F: \Delta \to \Delta \)
for \( f^{N} \) with return time function \( R \). Clearly this immediately implies also a GMY structure \( \tilde F : \Delta \to \Delta \) for \( f \) by simply taking a new return time function \( \tilde R = NR \). The integrability of \( R \) implies the integrability of \( \tilde R \) and this therefore implies Theorem~\ref{t.sing}.

We shall prove Theorem \ref{t.NUE}
in Sections \ref{sub.hyptimes} and  \ref{s.markov}.
In the remaining part of this section
we prove Proposition \ref{pr.NUE}. We first prove two auxiliary lemmas which are themselves of independent interest.

\begin{Lemma}\label{lem:neg}
 Let \( \mu \)
    be an ergodic regularly expanding acip.
For all sufficiently large \( N \) 
\begin{equation}\label{neg}
\int \log \|(Df^{N})^{-1}\| d\mu < 0.
\end{equation}
\end{Lemma}

\begin{proof}
By the integrability condition \eqref{regexp}, the
subadditive ergodic theorem, and condition  \eqref{exp} on the positivity of all Lyapunov exponents, 
there exists  \( \lambda>0\) such that for \( \mu \) almost every \( x \) we have 
 \begin{equation}\label{eq.tag} 
  \lim_{n\to\infty} \frac 1n\log \|Df^{n}(x)^{-1}\|=-\lambda.
  \end{equation}
  In fact this \(  \lambda  \) may be chosen precisely as the smallest Lyapunov exponent, see e.g. \cite[Addendum 4]{Boc}. We remark that since we are applying here the subadditive ergodic theorem we only have the inequality 
   \( \lim_{n\to\infty} \frac 1n\log \|Df^{n}(x)^{-1}\|\leq\int \log \|Df^{-1}\|d\mu \) and therefore this does not necessarily imply \( \int \log \|Df^{-1}\|d\mu <0 \). That is why we need to take some higher iterate~of~\( f \).

We define the sequence of sets
\[
B_{N}=\{x: \log \|Df^{N}(x)^{-1}\|>{-\lambda N}/{2}\}
 \]
 and write
 \begin{equation}\label{eq:splitint}
  \int \log \|(Df^{N})^{-1}\| d\mu =
 \int_{M\setminus B_{N}} \log \|(Df^{N})^{-1}\| d\mu + \int_{B_{N	}} \log \|(Df^{N})^{-1}\| d\mu
 \end{equation}
From \eqref{eq.tag} 
and the definition of \( B_{N} \) we have that \( \mu(B_{N})\to 0 \) as \( N\to \infty \)
and so for sufficiently large \( N \), assuming without loss of generality that \( \mu(M)=1 \),
we have
\begin{equation}\label{limited0}
    \int_{M\setminus B_N}\log\|Df^N(x)^{-1}\|d\mu \le -\frac\lambda2
N(1-\mu(B_N)) \leq -\frac \lambda 3 N.
\end{equation}
It is therefore sufficient to prove that the second integral on the right hand side of
\eqref{eq:splitint} is not too large. This is intuitively obvious since the measure of the \( B_{N} \) is going to zero, but we must make sure that this is not compensated by the fact that the integrand is possibly increasing in \( n \). We shall use the following

\begin{Sublemma}\label{critical} Let $\varphi \in L^1(\mu)$ and let $ (B_n)_n$ be a
sequence of sets with $\mu(B_n)\rightarrow 0$ as $n\to\infty$. Then
$$\frac{1}{n}\sum_{j=0}^{n-1}\int_{B_n}\varphi\circ f^j \,d\mu \to 0, \text{ as $n \to\infty$}.$$
\end{Sublemma}

\dem
From the $L^1$ Ergodic Theorem (see e.g.
\cite[Corollary~1.14.1]{W}) we have
 \begin{equation}\label{eq.no1}
  \frac{1}{n}\sum_{j=0}^{n-1}\varphi\circ f^j \stackrel{L^1}{\longrightarrow}\varphi^*, \quad\text{as $n\to\infty$}.
   \end{equation}
 Then we can write
\begin{eqnarray*}
  \left|\frac{1}{n}\sum_{j=0}^{n-1}\int_{B_n}\varphi\circ f^j \,d\mu - \int_{B_n}\varphi^* d\mu \right| &=& \left|\int\left(\frac{1}{n}\sum_{j=0}^{n-1}\varphi\circ f^j \,d\mu - \varphi^*\right)\chi_{B_n} d\mu \right| \\
    &\le& \int\left|\frac{1}{n}\sum_{j=0}^{n-1}\varphi\circ f^j \,d\mu - \varphi^*\right|d\mu.
\end{eqnarray*}
It follows from \eqref{eq.no1} that this last quantity converges to
0 when $n\to\infty$. Since we also have
$\int_{B_n}\varphi^* d\mu \rightarrow 0,\quad\text{when $n\rightarrow\infty$,}$ the conclusion then holds. \cqd

Returning to the proof of the Lemma,
by the chain rule we have
\begin{equation}\label{limited}
\int_{B_N}\log\|Df^N(x)^{-1}\|d\mu\leq
\sum_{j=0}^{N-1}\int_{B_N}\log\|Df(f^j(x))^{-1}\|d\mu
=: N b_N. 
\end{equation}
Applying Sublemma~\ref{critical} with \( \varphi = \log \|(Df)^{-1}\|
\) we get that  $b_N\rightarrow 0$ when $N\to\infty$.
%
%
Therefore, substituting  \eqref{limited0} and \eqref{limited} into \eqref{eq:splitint}
we obtain the desired conclusion.
\end{proof}

Next we identify possible candidates for the set A in the Proposition.
The existence of \( A \)
 is based on the following result on the
existence of finitely many ergodic components for powers of~\( f \).

\cle\label{le.ergdec} Let \(  \mu  \) be an ergodic invariant probability measure for \(  f  \). Given $N\ge 1$, there are $1\le \ell\le N$ and
$f^N$-invariant Borel sets $C_1,\dots,C_\ell$ such that:
  \begin{enumerate}
  \item $\{C_1,\dots,C_\ell\}$ is a partition ($\mu$-mod 0) of $M$ with $\mu(C_j)\ge 1/N$ for each $1\le j\le \ell$;
    \item $(f^N,\mu\vert C_j)$ is ergodic for each $1\le j\le \ell$.
  \end{enumerate}
\fle

\dem We start by proving that if $C$ is an $f^N$-invariant subset
with positive measure, then $\mu(C)\geq 1/N$. Indeed, assume by
contradiction that $\mu(C)< 1/N$. Consider the $f$- invariant set
$$\bigcup_{j=0}^{N-1}f^{-j}(C).$$
We have that
$$0<\mu\left(\bigcup_{j=0}^{N-1}f^{-j}(C)\right)\leq\sum_{j=0}^{N-1}\mu(f^{-j}(C))<1.$$
This gives a contradiction, because the set is $f$- invariant and
$\mu$ is ergodic.

Now, if $(f^N,\mu)$ is not ergodic, then we may  decompose $M$ into
a union of two $f^N$-invariant disjoint sets  with positive measure.
If the restriction of $\mu$ to some of these sets is not ergodic,
then we iterate this process. Note that this must stop after a
finite number of steps with at most $N$ disjoint subsets, since
$f^N$-invariant sets with positive measure have its measure bounded
from below by $1/N$. \cqd

For a given $N\ge 1$ we shall refer to the sets $A_i=\supp(\mu\vert C_i)$, with $1\le i\le\ell$ and  $A_1,\dots ,A_\ell$
given by the previous lemma, as the \emph{ergodic components} of
$(f^N, \mu)$. Observe that if $\mu$ is ergodic with respect to $f^N$,
then it has exactly one ergodic component.
We are now ready to complete the proof of the Proposition.

\dem[Proof of Proposition \ref{pr.NUE}]
Choosing \( N \) sufficiently large, Lemmas \ref{lem:neg} and
\ref{le.ergdec} imply that there  is some ergodic
component  \(A_{i} \) of \( (f^{N}, \mu) \) such that
\(
    \int_{A_i}\log \|(Df^{N})^{-1}\|d\mu < 0.
\)
Thus, by Birkhoff's Ergodic Theorem for $\mu$ almost every $x\in A_i$
one has
\begin{equation*}\label{eq.integral}
\lim_{n\rightarrow\infty}\frac{1}{n}\sum_{j=0}^{n-1}
\log\|Df^N(f^{Nj}(x))^{-1}\| =\int_{A_i}\log\|(Df^N)^{-1}\|d\mu <0.
 \end{equation*}
This proves NUE for \( f^{N}  \) for $\mu$ almost every point in the set \( A_i\).

Let us now prove the slow recurrence  condition SR for \( f^{N} \)
in the same ergodic component~$A_i$.  By assumption we have
$\log\dist(\cdot,\mathcal{C}_N) \in L^1(\mu)$. Therefore, by the
monotone convergence theorem we have
\( \int_{A_i}-\log\dist_\delta(\cdot ,
\mathcal{C}_N)d\mu\longrightarrow0, \)
when
$\delta\rightarrow0$.
So, by Birkhoff's Ergodic Theorem, given any $\epsilon>0$ there
exists $\delta>0$ such that
\begin{equation*}
\lim_{n\rightarrow\infty}\frac{1}{n}\sum_{j=0}^{n-1}-\log\dist_\delta(
f^{Nj}, \mathcal{C}_N) = \int_{A_i}-\log\dist_\delta(\cdot ,
\mathcal{C}_N)d\mu\leq\epsilon
\end{equation*}
for $\mu$ almost every $x \in A_i$ which is exactly conditions (SR).

\begin{Remark}\label{rem.int}
The integrability condition $\log\dist(\cdot,\mathcal{C}_N) \in L^1(\mu)$ is assumed explicitly in the statement of Theorem \ref{t.sing}, but not in the settings of Theorems \ref{t.locdif}, \ref{t.crit} and \ref{t.sing1}, where it actually follows from the other assumptions, in particular the nondegeneracy conditions on the critical and singular sets together with the integrability condition \( \log \|Df^{-1}\|\in L^{1}(\mu) \) stated in \eqref{regexp}.
Indeed, $\log\dist(\cdot,\mathcal{C}_N) \in L^1(\mu)$  is trivially satisfied in the setting of local diffeomorphisms. In the presence of critical points, condition (C1) implies \( \|Df(x)^{-1}\|^{-1} \leq d(x, \mathcal{C}_N)^{\beta'} \) which implies \( \|Df(x)^{-1}\|\geq d(x, \mathcal{C}_N)^{-\beta'} \) and thus \( \int \log \|Df^{-1}\|d\mu \geq -\beta '\int \log d(x, \mathcal{C}_N)d\mu  \). Then from condition \eqref{regexp} we get  \( -\int \log d(x, \mathcal{C}_N)d\mu < + \infty \) which gives
\( \int \log d(x, \mathcal{C}_N)d\mu > - \infty \) which implies $\log\dist(\cdot,\mathcal{C}_N) \in L^1(\mu)$. By a completely analogous argument, in the presence of singular points condition (C1') implies \( \|Df\|\geq d(x, \mathcal{C}_N)^{-\beta'} \) which implies
\( \log \|Df(x)\| \geq -\beta' \log d(x, \mathcal{C}_N) \)
 and   thus \( \log d(x, \mathcal{C}_N)\in L^1(\mu) \)  once again.
\end{Remark}
Concerning the set o points with dense orbits in \(  A_i  \), we know that almost all  orbits in the support of an ergodic measure have dense orbit in the support of the measure, see e. g. \cite[Proposition 4.1.18]{HK}, and therefore have full \(  \mu  \) measure in \(  A_i  \).

 To complete the proof we define \(  A'  \) as the set of points in \(  M  \) for which NUE  and SR hold and let \(  A:= A'\cap A_i  \). Then since both \(  A'  \) and \(  A_i  \) are \(  f^N  \) invariant, also \(  A  \) is \(  f^N  \) invariant. Moreover,  \(  \mu(A)=  \mu(A_i) > 0  \) and so, by absolute continuity, the Lebesgue measure of \(  A  \) is positive. Thus the set \(  A  \) satisfies the required properties.   This completes the proof of the Proposition.
\cqd

%
%
%

 \section{Choice of inducing domain}\label{sub.hyptimes}

We now begin the proof of Theorem~\ref{t.NUE}.  Assume that \( f \)
satisfies NUE and SR on a forward invariant set \( A \) with
positive Lebesgue measure subset of points whose orbit is dense in
$A$. This section is devoted to the proof of the  following result and its corollary, which will be used for the choice of our domain \( \Delta \) of
definition for the induced map \( F \).

\cpr \label{pr.balls}\label{l.dense}
For sufficiently small \( \delta_{1}>0 \), there is a ball $B$ of radius $\delta_1/4$
such that   $\leb(B\setminus A)=0$.
 Moreover, there are $p \in B$ and $N_0
\in\NN$ such that  $\bigcup_{j=0}^{N_0}f^{-j}\{p\}$ is $\delta_1/4$-dense
in $A$ and disjoint from the critical/singular set $\mathcal{C}$.
\fpr

A similar statement was proved in \cite[Lemma 5.6]{ABV} under some
 stronger assumptions in the definition of condition NUE.
We fix once and for all a point $p \in
A$ and \( N_{0}\in \mathbb N \) satisfying the conclusions of
Proposition~\ref{l.dense}, i.e. such that the set of
preimages of \( p \) up to \( N_{0} \) is \( \delta_{1}/4 \)-dense
in \( A \). For sufficiently small
\[ \delta_{0} \ll {\delta_{1}}
\]
where the conditions on \( \delta_{0} \) will be determined below,
we define the (Leb mod 0) subsets of $A$
\begin{equation}\label{disk}
\Delta = B(p,\delta_{0})\qand \Delta' = B(p,2\delta_{0}).
\end{equation}
As a relatively straightforward corollary of Proposition~\ref{pr.balls} we shall prove  that every ball of sufficiently large size, i.e. of radius at least \( \delta_{1} \), has a subset which maps diffeomorphically with bounded distortion onto \( \Delta' \) within a uniformly bounded number of iterations.

\cco \label{andadinha} If \( \delta_{0} \) is
sufficiently small, then
there are constants $D_0,K_0$ 
such that for any ball $\tilde B$ of radius $\delta_{1}$ with $\leb(\tilde B\setminus A)=0$ there are an open set $V\subset \tilde B$ and an
integer $0\leq m\leq N_0$ for which:
\begin{enumerate}
\item $f^m$ maps $V$ diffeomorphically onto $\Delta'$;
\item for each $x, y\in V$
$$
\log\left|\frac{\det Df^m(x)}{\det Df^m(y)}\right|\leq D_0
\dist(f^m(x), f^m(y));
$$
\item for each $0\leq j\leq m$
and for all $x\in f^{j}(V)$ we have
$$
K_{0}^{-1}\leq \|Df^{j}(x)\|, \|(Df^{j}(x))^{-1}\|, |\det
Df^{j}(x)|\leq K_0;
$$
in particular  \( f^{j}(V)\cap \cc = \emptyset \) .
\end{enumerate}
\fco

To prove Proposition~\ref{pr.balls} we introduce the fundamental notion of
a \emph{hyperbolic time} which will play a key role also in subsequent sections. This notion in the form in which we formulate it here was first defined and applied in \cite{Alv00}.
We fix once and for all $B>1$ and $\beta >0$ as in
Definition~\ref{criticalset}, and  take a constant $b > 0$ such that
$2b<\min\{1, \beta^{-1}\}$.

\cd Given $0<\sigma<1$ and $\delta > 0$, we say that $n$ is a
$(\sigma, \delta)$-{\em hyperbolic time\/} for $x\in M$ if for all
$1\le k \le n$,
\begin{equation}\label{d.hyperbolic1}
\prod_{j=n-k+1}^{n}\|Df({f^{j}(x)})^{-1}\| \le \sigma^k \qand
\dist_{\delta}(f^{n-k}(x), \mathcal{C})\geq \sigma^{bk}.
\end{equation}
In the case $\mathcal{C} = \emptyset$ the definition of hyperbolic
time reduces to the first condition in ~(\ref{d.hyperbolic1}).\fd
We denote
\[
H_j(\sigma, \delta)= \{x\in M :\text{ $j$ is a
$(\sigma, \delta)$-hyperbolic time for $x$}\}.
\]
A fundamental consequence of properties NUE and SR is the existence
of hyperbolic times as in the following result whose proof can be
found in \cite[Lemma 5.4]{ABV}.
\cle \label{existhiptimes}  There
are $\delta > 0$,
$0<\sigma<1$ and $\theta > 0$ 
such that
$$\limsup _{n\to\infty}\frac{1}{n}\#\{1\leq j\leq n: x\in H_j(\sigma, \delta)\} \geq
 \theta, $$
 for every $x \in A$.
\fle

From now on we consider \( \delta, \sigma, \theta \) fixed as in
Lemma~\ref{existhiptimes} and let \( H_{j}=H_{j}(\delta,
\sigma )\).

\cre \label{c.hyperbolic1} It easy to see that if $x \in H_j$ for a given $j\in\NN$, then
$f^i(x) \in H_m$ for any $1\le i<j$  and $m=j-i$ . 
 \fre
 The next lemma gives the main properties of the
hyperbolic times  such as uniform backward contraction and bounded
distortion. For the proof see \cite [Lemma~5.2, Corollary~5.3]{ABV}.

\cle \label{l.contraction} There exists $\delta_1, C_{1}>0$
such that if $n$ is a $(\sigma, \delta)$-hyperbolic time for $x$,
then there is neighborhood $V_n$ of $x$ such that:
\begin{enumerate}
\item $f^{n}$ maps $V_n$ diffeomorphically onto a
ball of radius $\delta_1$ around  $f^{n}(x)$;
\item \label{contration} for every $1\le k
\le n$ and $y, z\in V_n$, $$ \dist(f^{n-k}(y),f^{n-k}(z)) \le
\sigma^{k/2}\dist(f^{n}(y),f^{n}(z));$$
\item \label{p.distortion} for any
$y,z\in V_n$
$$
\log \frac{|\det Df^{n} (y)|}
                     {|\det Df^{n} (z)|}
            \le C_1 \dist(f^{n}(y),f^{n}(z)).
$$
\end{enumerate}
\fle
We call the sets \( V_n \)  \emph{hyperbolic pre-balls} and
their images \( f^{n}(V_n) \)  \emph{hyperbolic balls}. The latter are actually balls of radius \( \delta_1>0 \). Notice that  $\delta_1>0$ can be taken arbitrarily small for a fixed choice of $\delta>0$. 

\cle \label{co.contraction} Assume that $2\delta_1<\delta<1$. There is $C_2>0$ such that if $n$ is a $(\sigma, \delta)$-hyperbolic time for $x$ and $V_n$ is the corresponding hyperbolic pre-ball, then:
\begin{enumerate}
\item for every $1\le k
\le n$ and $y\in V_n$, $$ \dist(f^{k}(y),\cc) \ge \frac12
\min\{\delta, \sigma^{b(n-k)}\} ;$$
\item  for any
Borel sets $Y, Z\subset V_n$,
$$
\frac{1}{C_2}\frac{\leb(Y)}{\leb(Z)}\leq
\frac{\leb(f^n(Y))}{\leb(f^n(Z))} \leq C_2\frac{\leb(Y)}{\leb(Z)};
$$
 \item there is  $\tau_n >0$ such that for any $x\in
H_n$ one has $B(x,{\tau_n}) \subset V_n$. In particular, every \( H_{n} \) is covered by a finite number of hyperbolic pre-balls.
\end{enumerate}
\fle
\dem
Since $n$ is a hyperbolic time for $x$,  then
 using the second item of Lemma~\ref{l.contraction} we obtain
  \begin{eqnarray}
  \dist (f^{k}(y),\cc)
  &\ge &\dist(f^{
  k}(x),\cc)-\dist(f^k(x),f^k(y))\nonumber\\
  &\ge &\dist(f^{k}(x),\cc)-\delta_1 \sigma^{(n-k)/2}\label{eq.chato}
  \end{eqnarray}
  Now, if $\dist(f^{k}(x),\cc)=\dist_\delta(f^{k}(x),\cc)$, recalling that we have taken $b<1/2$, then
  using~\eqref{eq.chato}  and the definition of hyperbolic time we get
   $$\dist (f^{n-k}(y),\cc) \ge \frac12  \sigma^{b(n-k)},$$
  as long as $\delta_1<1/2$. Otherwise, we have $\dist(f^{k}(x),\cc)\ge \delta$, and so
  $$\dist (f^{k}(y),\cc) \ge \frac12  \delta,$$
  as long as $\delta_1<\delta/2$.
  This proves the first item.

  Let us now prove the second item.
  By a change of variables induced by $f^n$ we may write
\begin{eqnarray*}\displaystyle
  \frac{\leb(f^n(Y))}{\leb(f^n(Z))} &=& \frac{\int_{Y}|\det Df^n(y)| d\leb(y)}{\int_{Z} |\det Df^n(z)| d\leb(z)}\\
   &=& \frac{|\det Df^n(y_0)|\int_{Y}\left|\frac{\det Df^n(y)}{\det Df^n(y_0)}\right| d\leb(y)}
   {|\det Df^n(z_0)|\int_{Z}\left|\frac{\det Df^n(z)}{\det Df^n(z_2)}\right| d\leb(z)},
\end{eqnarray*}
where   $y_0$ and $z_0$ are chosen arbitrarily in $ Y$ and $ Z$,
respectively. Using the third item of Lemma~\ref{l.contraction} we
easily find uniform bounds for this expression.

To prove the third item, we observe that from the first item we can
take a neighborhood $\mathcal{N}_n$ of the critical/singular set  such that
$V_n\subset  M\setminus \mathcal N_n$. Hence, there is a constant
$K_n$  depending only on the hyperbolic time $n$ such that $
\|Df^n|_{V_n}\|\leq K_n $ and so the result follows.\cqd


 \dem[Proof of Proposition \ref{pr.balls}]
 For the first part of the Proposition,
 it is enough to prove that there exist balls of radius
$\delta_1/4$ where the relative measure of $A$ is arbitrarily close
to one. Since the set of points with infinitely many  hyperbolic
times is positively invariant and $A$ also is positively invariant,
we may assume, without loss generality, that every point in $A$ has
infinitely many hyperbolic times. Let $\epsilon >0$ be some small
number. By regularity of $\leb$, there is a compact set $A_c\subset
A$ and open set $A_0\supset A $ such that
\begin{equation}\label{m.regular}
\leb(A_0\setminus A_c) <\epsilon\leb(A ).
\end{equation}
Assume that $n_0$ is large enough so that for every $x\in A_c$, any
hyperbolic preball $V_n(x)$ with $n\geq n_0$ is contained in $A_0$.
Let $W_n(x)$ be a part of $V_n(x)$ that is sent diffeomorphically by
$f^n$ onto the ball $B(f^n(x), {\delta_1/4 })$. By compactness there
are $x_1, \ldots, x_r \in A_c$ and $n(x_1), \ldots, n(x_r)\geq n_0$
such that
\begin{equation}\label{cobertura}
A_c \subset W_{n(x_1)}(x_1)\cup\ldots \cup W_{n(x_r)}(x_r).
\end{equation}
For the sake of notational simplicity we shall write for each $1\leq
i \leq r$
$$
V_i=V_{n(x_i)}(x_i), \quad W_i= W_{n(x_i)}(x_i) \qand n_i=n(x_i).
$$
Let $n^*_1< n^*_2< \ldots< n^*_s$ be the distinct values taken by the $n_i$'s.
Let $I_1\subset\NN$ be a maximal subset of $\{1,\ldots, r\}$ such
that for each $i \in I_1$ both $n_i=n^*_1$, and $W_i\cap
W_j=\emptyset$ for every $j\in I_1$ with $j\neq i$. Inductively, we
define $I_k$ for $2\leq k \leq s$ as follows: supposing that
$I_1,\ldots, I_{k-1}$ have already been defined, let $I_k$ be a
maximal set of $\{1,\ldots,r\}$ such that for each $i \in I_k$ both
$n_i=n^*_k$, and $W_i\cap W_j = \emptyset$ for every $j \in
I_1\cup\ldots\cup I_k$ with~$i \neq j$.

Define $I= I_1\cup\ldots\cup I_S$. By construction we have that
$\{W_i\}_{i\in I}$ is a family of pairwise disjoint sets. We claim
that $\{V_i\}_{i\in I}$ is a covering of $A_c$. To see this, recall
that by construction , given any $W_j$ with $1\leq j\leq r$, there
is some $i \in I$ with $n(x_i)\leq n(x_j)$ such that $W_{x_j}\cap
W_{x_i}\neq\emptyset$. Taking images by $f^{n(x_i)}$ we have
$$
f^{n(x_i)}(W_j)\cap B(f^{n(x_i)}(x_i), {{\delta_1}/4})\neq\emptyset.
$$
It follows from  Lemma \ref{l.contraction}, item (\ref{contration})
that
$$
\diam(f^{n(x_i)}(W_j))\leq\frac{\delta_1}{2}\sigma^{(n(x_j)-n(x_i))/2}\leq\frac{\delta_1}{2},
$$
and so
$$
f^{n(x_i)}(W_j)\subset B(f^{n(x_i)}(x_i), {{\delta_1}}).
$$
This gives that $W_j\subset V_i$. We have proved that given any
$W_j$ with $1\leq j \leq r$, there is $i \in I$ so that $W_j\subset
V_i$. Taking into account (\ref{cobertura}), this means that
$\{V_i\}_{i\in I}$ is a covering of $A_c$.

 By Lemma \ref{l.contraction}, item  (\ref{p.distortion}) one may find
$\tau>0$ such that
$$
\leb(W_i)\geq \tau \leb(V_i), \quad \text{for all } i \in I.
$$
Hence,
\[ 
       \nonumber  \leb \left(\bigcup_{i \in I}W_i\right) = \sum_{i \in I}\leb(W_i) 
        \nonumber  \geq \tau\sum_{i \in I}\leb(V_i) 
        \nonumber \geq \tau \leb\left(\bigcup_{i \in I}V_i\right) 
        \nonumber  \geq\tau \leb(A_c).
\]
From \eqref{m.regular} one deduces that
$\leb(A_c)>(1-\epsilon)\leb(A )$. Noting that the constant $\tau$
does not depend on $\epsilon$, choosing $\epsilon>0$ small enough we
may have
\begin{equation}\label{desigual1}
\leb\left(\bigcup_{i \in I}W_i\right)>\frac{\tau}{2}\leb(A ).
\end{equation}
We are going to prove that
\begin{equation}\label{desigual2}
    \frac{\leb (W_i\setminus A)}{\leb(W_i)}< \frac{2\epsilon}{\tau},
    \quad \text{for some } i\in I.
\end{equation}
This is enough for our purpose, since taking $B= f^{n(x_i)}(W_i)$ we
have by invariance of $A$ and Lemma \ref{l.contraction}, item
(\ref{p.distortion})
$$
\frac{\leb(B\setminus
A)}{\leb(B)}\leq\frac{\leb(f^{n(x_i)}(W_i\setminus
A))}{\leb(f^{n(x_i)}(W_i)}\leq C_0\frac{\leb(W_i\setminus
A)}{\leb(W_i)}= \frac{2 C_0\epsilon}{\tau},
$$
which can obviously be made arbitrarily small. From this one easily
deduces that there are disks of radius $\delta_1/4$ where the
relative measure of $A$ is arbitrarily close to one.

Finally, let us prove (\ref{desigual2}). Assume, by contradiction,
than it does not hold. Then, using (\ref{m.regular}) and
(\ref{desigual1})
\[
  \epsilon\leb(A) >\leb(A_0\setminus A_c) 
\geq \leb \left(\left(\bigcup_{i\in I}W_i\right)\setminus A\right) 
\geq\frac{2\epsilon}{\tau}\leb\left(\bigcup_{i\in I}W_i\right) 
\geq \epsilon\leb(A).
\]
This gives a contradiction and proves the first part of the Proposition.

For the second part,  since we are assuming that $f$ is
a local diffeomorphism up to a set of zero Lebesgue measure, then
the  set
$$
\mathcal{B}=\bigcup_{n \geq 0}f^{-n}\left(\bigcup_{m \geq
0}f^{m}(\mathcal{C})\right)
$$
has Lebesgue measure equal to zero. On the other hand,  there is a positive Lebesgue
measure subset of points in $A$ with dense orbit. Thus there must be
some point $q\in A\setminus \mathcal{B}$ with dense orbit in $A$.
Take $N_0\in\NN$ for which $q, f(q), \ldots , f^{N_0}(q)$ is
$\delta_1/4$-dense in $A$ and $f^{N_0}(q)\in B$. The point $p=f^{N_0}(q)$
satisfies the conclusion of the lemma. \cqd

 \dem[Proof of Corollary \ref{andadinha}]  The proof is similar to \cite[Lemma 2.6]{ALP},
though we repeat it here  in order
to clarify the fact that also it holds in the situation where \( A
\) is not necessarily equal to \( M \).
 Since $\cup_{j=0}^{N_0} f^{-j}\{p\} $ is
 disjoint from $\mathcal{C}$,
 then choosing $\delta_0$ sufficiently small we have
that each connected component of the preimages of $B(p,2\delta_0)$
up to time $N_0$ is bounded away from the critical/singular set $\mathcal{C}$
and is contained in a ball of radius ${\delta_{1}/4}$. Moreover,
$\cup_{j=0}^{N_0} f^{-j}\{p\} $ is \( \delta_{1}/4 \)-dense in \( A
\) and this immediately implies that any ball $B$ as in the statement of the lemma contains a preimage
$V$ of $B(p,2\delta_0)$ which is mapped diffeomorphically onto
$B(p,2\delta_0)$ in most $N_0$ iterates, thus giving (1).  Moreover,
since the number of iterations and the distance to the critical
region are uniformly bounded, we immediately get (2) and (3). \cqd

\section {Markov structure and recurrence times}
\label{s.markov}

In this section we give the complete construction of the induced map
\( F: \Delta \to \Delta \) and prove the required properties. We divide the section into four subsections. In \ref{se.algo} we give the purely combinatorial algorithm for the construction. In
\ref{se:expdist} we show that all partition elements constructed according to this algorithms satisfy the required expansion and distortion properties. In \ref{se:sat} we show that the algorithm actually gives a partition mod 0 of \( \Delta \) in the sense that Lebesgue almost every point of \( \Delta \) belongs to the interior of some partition element. Finally, in \ref{s.inducing} we prove the integrability of the return times.

\subsection{The partitioning algorithm}\label{se.algo}

We start by describing an inductive construction of the $\cp$ partition
($\leb $ mod 0) of $\Delta $.
Given a point $x \in H_{n}\cap\Delta$, by
Lemma~\ref{l.contraction} there exists a hyperbolic pre-ball $V_{n
}(x)$ such that $
 f^{n}(V_{n}(x))= B(f^{n}(x),{\delta_1}).
$ From Proposition~\ref{andadinha}, there are a set $V \subset
B(f^{n}(x),{\delta_1})$ and an integer $0\leq m\leq N_0$ such that
$f^{m }(V)=\Delta'\supset\Delta$. Define
\begin{equation}\label{candidate}
U_{n,m }^{x} = (f|_{V_{n}(x)} ^{n,m})^{-1}(\Delta).
\end{equation}
These sets $U^x_{n,m}$ are the candidates for elements of the
partition of \( \Delta \) corresponding to the induced map \( F \)
since they are mapped onto \( \Delta \) with uniform expansion and
bounded distortion. Notice that the sets $U^{x}_{n,m}$ and \(
U^{x'}_{n',m'} \) for distinct points \( x, x' \) are not
necessarily disjoint and this is a major complication in the
construction. The strategy for dealing with this is additionally
complicated by the fact that $U^{x}_{n,m}$ does not necessarily
contain the point $x$. To deal with these issues,
we introduce sets $\Delta_n$ and ${S}_n$ such that $\Delta_n$ is the
part of $\Delta$ that has not been partitioned up to time $n$, and ${S}_n$, that we
call the \emph{satellite set}, corresponding to the portion of a
reference hyperbolic pre-ball that was not used for constructing an
element of the partition. It is important to note that to each step
of the algorithm is associated a unique hyperbolic time and possibly
several distinct return times.

\subsubsection*{First step of induction}
We fix  some large \(n_0\in\NN \)  and ignore any dynamics occurring
up to time \( n_{0} \).
Define $\Delta^c=M\setminus\Delta$.
By  the third item of Lemma~\ref{co.contraction}, \( H_{n_{0}} \) can be covered by a finite number of hyperbolic pre-balls, and thus there are $z_1,\dots, z_{N_{n_0}}\in H_{n_0}$ such that
 $$H_{n_0}\subset V_{n_0}(z_1)\cup\cdots \cup V_{n_0}(z_{N_{n_0}}).$$
Consider  a maximal family
$$
\mathcal{U}_{n_0}=\{U_{n_0,m_1}^{x_1},\ldots,
U_{n_0,m_{k_{n_0}}}^{x_{k_{n_0}}}\}
$$
of pairwise disjoint sets of type (\ref{candidate}) contained in
$\Delta$ with $x_1,\dots,x_{k_{n_0}}\in\{z_1,\dots, z_{N_{n_0}}\}$. These are the elements of the partition $\cp$ constructed
in the
$n_0$-step  of the algorithm. 
Set $R(x)=n_0+m_i$ for each $x\in
U^{x_i}_{n_0,m_i}$ with $0\leq i \leq k_{n_0}$.
Now let
$$\widetilde H_{n_0}=\{z_1,\dots, z_{N_{n_0}}\}
\setminus\{x_1,\dots, x_{k_{n_0}}\}.$$
be the set of points in $\{z_1,\dots, z_{N_{n_0}}\}$ which
were not ``used'' in the construction of $\mathcal{U}_{n_0}$. Notice that the reason they were not used is that the associated sets of the form \( U_{n_{0}, m} \) overlap one of the sets in \( \mathcal U_{n_{0}} \) which were selected. We want to keep track of which sets overlap which and so, for each given  $U\in
\mathcal{U}_{n_0}$, and each $0\le m\le N_0$, we define
\begin{equation}\label{auxiliosatelite0}
    H_{n_0}^m(U)=\left\{x \in \widetilde H_{n_0}: U^x_{n_0,m}\cap U \neq\emptyset\right\}
\end{equation}
and the  \emph{$n_0$-satellite}
\begin{equation*}
    S_{n_0}(U)= \bigcup_{m=0}^{N_0}\bigcup_{x\in
    H_{n_0}^m(U)}V_{n_0}(x)\cap (\Delta\setminus U).
\end{equation*}
Thus, the  \( n_{0} \)-satellite of \( U \) is the union of all hyperbolic pre-balls which ``could have'' had a subset returning to \( \Delta \) but were unlucky in that such a subset overlaps the set \( U \) which was chosen instead.
It will be convenient to consider also the $n_0$-satellite
associated to $\Delta^c$ 
\begin{equation*}
    S_{n_0}(\Delta^c)= \bigcup_{m=0}^{N_0}\bigcup_{x\in
    H_{n_0}^m(\Delta^c)}V_{n_0}(x)\cap \Delta.
\end{equation*}
Finally we define the \emph{global $n_0$-satellite}
\begin{equation}\label{satelite_n0}
{S}_{n_0} = \bigcup_{U\in \mathcal{U}_{n_0}}{S}_{n_0}(U) \cup
{S}_{n_0}(\Delta^c)
\end{equation}
and
\begin{equation}\label{delta_n0}
\Delta_{n_0} = \Delta\setminus \bigcup_{U \in \mathcal{U}_{n_0}}U.
\end{equation}

\subsubsection*{General step of induction}
The general step of the construction follows the ideas above with
minor modifications. Assume that the set $\Delta_{s}$ is defined for
each \( s\leq n-1 \). Once more by  the third item of
Lemma~\ref{co.contraction} there are $z_1,\dots, z_{N_{n}}\in H_{n}$
such that $$H_{n}\subset V_{n}(z_1)\cup\cdots \cup
V_{n}(z_{N_{n}}).$$ Consider a maximal family
$$
\mathcal{U}_{n}=\{U_{n,m_1}^{x_1},\ldots,
U_{n,m_{k_n}}^{x_{k_n}}\}
$$
of pairwise disjoint sets of type (\ref{candidate})~contained in
$\Delta_{n-1}$  with $x_1,\dots,x_{k_{n}}\in\{z_1,\dots, z_{N_{n}}\}$. These are the elements of the partition $\cp$
constructed in the $n$-step  of algorithm. Set $R(x)=n+m_i$ for each
$x\in U^{x_i}_{n,m_i}$ with $0\leq i \leq k_n$.
Let
$$\widetilde H_{n}=\{z_1,\dots, z_{N_{n}}\}\setminus\{x_1,\dots, x_{k_{n}}\}.$$
Given $U\in
\mathcal{U}_{n_0}\cup\cdots\cup \mathcal{U}_n$, we define for $0\le m\le N_n$
\begin{equation}\label{auxiliosatelite}
    H_{n}^m(U)=\left\{x \in \widetilde H_{n}: U^x_{n,m}\cap U \neq\emptyset\right\}
\end{equation}
and its \emph{$n$-satellite}
\begin{equation*}
    S_{n}(U)= \bigcup_{m=0}^{N_n}\bigcup_{x\in
    H_{n}^m(U)}V_{n}(x)\cap (\Delta\setminus U).
\end{equation*}
It will be convenient to consider also the \emph{$n$-satellite}
associated to $\Delta^c$
\begin{equation*}
    S_{n}(\Delta^c)= \bigcup_{m=0}^{N_n}\bigcup_{x\in
    H_{n}^m(\Delta^c)}V_{n}(x)\cap (\Delta\setminus \Delta^c).
\end{equation*}
Finally we define the \emph{global $n$-satellite}
\begin{equation}\label{satelite_n}
{S}_{n} = \bigcup_{U\in \mathcal{U}_{n_0}\cup\cdots\cup\,
\mathcal{U}_n}{S}_{n}(U) \cup {S}_{n}(\Delta^c)
\end{equation}
and
\begin{equation}\label{delta_n}
\Delta_{n} = \Delta\setminus \bigcup_{U\in
\mathcal{U}_{n_0}\cup\cdots\cup\, \mathcal{U}_n
}U.
\end{equation}

\cre\label{re.saturacao}
Note that the construction of these objects has been performed in such a way that for each $n\ge n_0$ one has
 $$H_n\subset S_n\cup\bigcup_{U\in \mathcal U_{n_0}\cup\cdots\mathcal U_n }U.$$
\fre

\subsection {Expansion and bounded distortion}\label{se:expdist}

Recall that, by construction, the return time $R$ for an element $U$
of the partition $\cp$ of $\Delta$ is made by a certain number $n$
of iterations given by the hyperbolic time of a pre-ball
$V_{n}\supset U$, plus a certain number $m\leq N_{0}$ of additional
iterates which is the time it takes to go from $f^{n}(V_{n})$, which
could be anywhere in $M$, to $f^{n+m}(V_{n})$, which covers $\Delta$
completely. It follows from Lemmas ~\ref{l.contraction} and
Proposition~\ref{andadinha} that
\[
\|Df^{n+m}(x)^{-1}\| \le \|Df^{m}(f^{n}(x))^{-1}\|.
\|Df^{n}(x)^{-1}\| 
   \le K_{0}\sigma^{n/2} 
   \le K_{0}\sigma^{(n_{0}-N_{0})/2}.
\]
Taking $n_{0}$ sufficiently large we can make this last expression
smaller than one.
We also need to show that there exists a constant $K > 0$ such that
for any $x, y$ belonging to an element $U \in \cp$ with return time
$R$, we have
$$\log\left|\frac{\det Df^{R}(x)}{\det Df^{R}(y)}\right|\leq K \dist(f^{R}(x),
f^{R}(y)).$$ By Lemmas~\ref{l.contraction} and Proposition~\ref{andadinha},
it is enough to take $K = D_{0}+ C_{1}K_{0 }$.

\cre Analogously to Lemma~\ref{co.contraction}, there exists a constant $C_4 >0$ such that for any Borel sets $Y, Z
\subset \left(f|_{V_n }^{n+m}\right)^{-1}(\Delta')$ we have
$$
\frac{1}{C_4}\frac{\leb(Y)}{\leb(Z)}\leq
\frac{\leb(f^{n+m}(Y))}{\leb(f^{n+m}(Z))} \leq
C_4\frac{\leb(Y)}{\leb(Z)}.
$$ \fre

\subsection {The measure of satellites}\label{se:sat}

In this section, we will show that the algorithm described above does indeed produce a
partition ($\leb$ mod 0) of $\Delta$.
Notice first of all that since $\Delta\supset\Delta_{n_0}\supset \Delta_{n_0+1}\supset...$,
we only have to check that $\leb (\cap_{n}\Delta_{n})=0$. This is a consequence of the following

 \cpr \label{prop.Sn}
$\sum_{n=n_0}^{\infty}\leb({S}_{n}) <\infty$. \fpr

Indeed, it follows from Proposition \ref{prop.Sn} and the
Borel-Cantelli Theorem that Lebesgue almost every point in $\Delta$
belongs to finitely many ${S}_{n}'s$. Since a generic point
$x\in\Delta$ has infinitely many $\sigma$-hyperbolic times, it follows that for
almost every $x\in \Delta$ one can find $n$ such that $x\in H_{n}$
and $x\notin {S}_j$ for $j\ge n$.
Thus, recalling Remark~\ref{re.saturacao} one must have $x\in \{R=n+m\}$ for some
$0\leq m\leq N_0$. Since this is valid for Lebesgue almost all $x\in\Delta$,
then $\leb (\cap_{n}\Delta_{n})=0$.

Thus we just need to prove Proposition \ref{prop.Sn}. We shall prove first two auxiliary lemmas.
The first one gives in particular that $ U_{n,m }^{x} $ represents a positive proportion of $ V_{n}(x) $.

 \cle\label{estpreball} There exists
$C_{3}>0$ such that given any $n\ge 1$  and  any set of points $x_1,\dots ,x_N\in H_n$ such that  the corresponding \( U_{n,m}^{x_i} \) coincide, i.e. $U_{n,m}^{x_i}=U_{n,m}^{x_1}$ for $1\le i\le N$, then
\begin{equation*}\label{ultimah}
    \leb\left(\bigcup_{i=1}^{N} V_n(x_i)\right)\leq C_3\leb(
    U^{x_1}_{n,m}).
   \end{equation*}
\fle
 \dem
 For simplicity of notation we shall write for  $1 \leq i\leq
 N$,
 $$V_n(x_i)=
 V_i\qand  B(f^n(x_i), {\delta_1})= B_i .
$$
We define
$$
X_1= V_1 \qand X_{i}= V_i\setminus\bigcup_{j=1}^{i-1} V_j, \quad \text{for $2\le i\le N$}.
$$
Similarly
$$
Y_1= B_1 \qand Y_{i}= B_i \setminus\bigcup_{j=1}^{i-1} B_j, \quad \text{for $2\le i\le N$}.
$$
 Observe that
$
V_1 \cup \cdots \cup V_N=X_1 \; \dot{\cup} \cdots \dot{\cup}\; X_N$ and $
B_1\cup \cdots\cup B_N =  Y_1 \; \dot{\cup} \cdots \dot{\cup}\; Y_N.
$
Recalling that $U_{n,m}^{x_i}=U_{n,m}^{x_1}$ for $1\le i\le N$, by bounded distortion we have
\begin{equation*}
    \frac{\leb (X_i)}{\leb (U^{x_1}_{n,m}) }\leq C_2 \frac{\leb (Y_i)}{\leb f^n(
    U^{x_1}_{n,m})}. 
\end{equation*}
Hence
\[
 \frac{\leb(V_1 \cup \ldots \cup V_N)}{\leb (U^{x_1}_{n,m})}
  =
  \frac{\sum_{i=1}^{N}\leb (X_i)
   }{\leb ( U^{x_1}_{n,m})}
  \leq C_2\frac{\sum_{i=1}^{N}\leb (Y_i)
   }{\leb (f^n( U^{x_1}_{n,m}))} 
   \leq C_2\frac{\leb (B_1\cup \ldots \cup B_N)}{\leb (f^n(
    U^{x_1}_{n,m}))}.
\]
    Moreover, by a change of variables we have
$$
\int_{\Delta}d\leb(z )= \int_ {f^n(U^{x}_{n,m})}|\det
Df^m(y)|d\leb(y).
$$
 So, by Proposition~\ref{andadinha},
$$
\leb(f^n(U^{x}_{n,m}))\geq K_0^{-1}\leb(\Delta).
$$
Therefore the result follows taking $ C_3=C_2K_0{\leb (M)}\backslash
\leb(\Delta). $ \cqd

The next lemma shows that, for each $n$ and $m$ fixed, the Lebesgue
measure of the union of candidates $U^{x}_{n,m}$ which intersects an
element of partition is proportional to the Lebesgue measure of this
element. The proportion constant can actually be made uniformly
summable in $n$.

 \cle\label{estimativas}There
exists  $C_5> 0$ such that given $0\le m\le N_0$, $k\ge n_0$ and $U\in
\mathcal{U}_k$, then for any $n\ge k$ 
$$
\leb \left(\bigcup_{x \in
H_n^m(U)}U_{n,m}^x\right)\leq C_5\sigma^{\frac{n-k}{2}} \leb(U).
$$
\fle

\dem Consider an integer $k\ge n_0$ and a set $U\in \mathcal{U}_k$. Recall that by construction we have $R\vert U=k+m_0$ for some
 $0\le m_0\le N_0$.  Moreover, $U$ is part of a hyperbolic preball $V_k$  which is sent diffeomorphically onto $\Delta$ by $f^{k+m_0}$; recall~\eqref{candidate}. We
define
$$T= (f|_{V_{k}} ^{k+m_0})^{-1}(\Delta'\setminus \Delta).$$
Now consider
$L\in \NN$ large so that
$$
 2K_0\sigma^{(L-1)N_0/2}< 1,
$$
where $N_0$ and $K_0$ are given by Proposition~\ref{l.dense} and
Proposition~\ref{andadinha} respectively.
We shall split the proof into two separate cases depending on whether $k \leq n\leq k+LN_0$ or $n> k+LN_0$.\smallskip

(1) Assume first that $k \leq n\leq k+LN_0$. Fix some set $U^x_{n,m}$ with $x \in
H_n^m(U)$.
%
%
Having in mind the conclusion we need, it is no restriction to assume that  $U^x_{n,m}$ intersects the complement of $U\cup T$ (observe that $U\cup T$ is obviously a proportion of $U$). Hence, there is
a point $u\in U_{n,m}^x\cap T$ 
for which $v=f^{k+m_0}(u)$ satisfies $\dist(v,p)=3\delta_0/2$.
\medskip

\noindent{\sc Claim:}  \emph{There is a uniform constant $\tilde\rho>0$ for which $\leb(f^{n,m}(U_{n,m}^x\cap T))\ge \tilde\rho.$}
\medskip

\noindent Consider first that $n+m>k+m_0$. Considering $\ell=
(n+m)-(k+m_0)$ we have
 $0\le \ell\le LN_0+N_0.$
Just by continuity there is $\rho>0$ and a neighborhood $V_\rho$ of $u$ such that both
 $$
f^{n+m}(V_\rho)=B(f^\ell(v),\rho)\cap \Delta\qand f^{k+m_0}(V_\rho)\subset \Delta'\setminus\Delta.
$$
Observe that $f$ sends $f^{k+m_0}(V_\rho)$ onto
$B(f^\ell(v),\rho)\cap \Delta$ in $\ell$ iterates. Moreover, when we
look back, we see that the $\ell$ backward iterates comprise a
certain number of at most $N_0$ backward iterates plus at most
$LN_0$ backward iterates of a hyperbolic ball. Thus, by
Proposition~\ref{andadinha} and Lemma~\ref{co.contraction} we
guarantee some uniform bound on the derivative of those backward
iterates. This means that it is possible to choose $\rho$ uniformly.
Hence, there $\tilde\rho$ (depending only on $\rho$) for which
 $$\leb(f^{n+m}(U_{n,m}^x\cap T))\ge \tilde\rho,$$
which gives the claim in this case.

Consider now  $n+m<k+m_0$. Taking in this case $\ell=(k+m_0)-(n+m)$, we have $\ell\le N_0$.
By continuity there is $\rho>0$ for which
$$
f^\ell(B(f^{n+m}(u),\rho)\cap \Delta)\subset \Delta'\setminus\Delta.
$$
By Proposition~\ref{andadinha} we have some uniform bound on the
derivative of the backward iterates of
$f^\ell(B(f^{n+m}(u),\rho)\cap \Delta)$. This means that it is
possible to choose $\rho$ uniformly, and so there exists $\tilde\rho$
(depending only on $\rho$) for which the claim again holds.

\medskip
Let us now use the claim to prove the first part of the lemma. Note that we can find $\xi>0$ such that
$$
\frac{\leb(\Delta'\setminus \Delta)}{\leb(\Delta)} \leq \xi,
$$
and so,  by bounded distortion in time $k+m_0$,
\begin{equation}\label{eq.T/U}
\frac{\leb(T)}{\leb(U)}\le C_4\frac{\leb(\Delta'\setminus
\Delta)}{\leb(\Delta)}\le C_4 \xi.
\end{equation}
By bounded distortion 
in the time $n+m$,
\begin{equation}\label{eq.U/TU}
\frac{\leb (U_{n,m}^x)}{\leb(U_{n,m}^x\cap T)}\le C_4
\frac{\leb(\Delta)}{\tilde\rho}.
\end{equation}
Now observe that for fixed $n,m$ and given $x,x'\in H_n^m(U)$ one must have
$U^{x}_{n,m}=U^{x'}_{n,m}$ or $U^{x}_{n,m}\cap U^{x'}_{n,m}=\emptyset$.
Since $U_{n,m}^x\cap T$ is contained in $T$, it follows from \eqref{eq.T/U}
and \eqref{eq.U/TU} that
$$
\leb\left(\bigcup_{x \in H_n^m(U)}U_{n,m}^x\right)\leq C_4^2\xi
\frac{\leb(\Delta)}{\tilde\rho}\leb(U).
$$

\smallskip

(2) Assume now that $n> k+LN_0$. Since for each
$U_{n,m}^x$ we have
$$
\diam( f^{k+m_0}(U_{n,m}^x))\leq 2\delta_0K_0\sigma^{n-(k+m_0)/2}
$$
then, by the choice of $L$, the sets $U_{n,m}^x$ are contained
in $T \cup U$ for  $n> k+LN_0$. Moreover, defining the annulus inside $\Delta'$ around the boundary of $\Delta$
 $$A_{n,k}=\{x \in \Delta' : \dist(x,\partial\Delta)\leq
    2\delta_0K_0\sigma^{n-(k+N_0)/2}
    \}
    $$
we have
 $$f^{k+m_0}(U_{n,m}^x)\subset A_{n,k}.$$
%
By bounded distortion
\begin{eqnarray}
\frac{\leb\left(\bigcup_{x \in H_n^m(U)}U_{n,m}^x\right)}{\leb (U)}&\leq& C_4
\frac{ \leb(A_{n,k})}{ \leb(\Delta)}.\nonumber  \label{est1}
\end{eqnarray}
Since there is a constant $\eta>0$ for which
 $$
\frac{\leb(A_{n,k})}{\leb(\Delta)}\leq \eta
\sigma^{\frac{n-k}{2}},
 $$
then we have
$$
\leb \left(\bigcup_{x \in H_n^m(U)}U_{n,m}^x\right)\leq \eta
C_4\sigma^{\frac{n-k}{2}}\leb (U).
$$
Take $C_5=\max \{C_4^2\xi {\leb(\Delta)}/{\tilde\rho}, \eta
C_4\}$. \cqd

We are now ready to complete the proof of the main proposition of this section.

\dem[Proof of Proposition \ref{prop.Sn}]
 Observe that
\begin{equation}\label{eq.sumsum}
\sum_{n=n_0}^{\infty}\leb({S}_{n}) \le
\sum_{n=n_0}^{\infty}\leb({S}_{n}({\Delta^c }))+
\sum_{k=n_0}^\infty\sum_{U\in
\mathcal{U}_k}\sum_{n=k}^{\infty}\leb({S}_{n}(U)).
\end{equation}
We start by estimating  the sum with respect to the satellites of $\Delta^c$. It follows from the definition of
${S}_{n}({\Delta^c})$ and Lemma~\ref{l.contraction} that
$$
{S}_{n}({\Delta^c})\subset \{x \in \Delta:\,
\dist(x,\partial\Delta)<2\delta_1\sigma^{n/2}\}.
$$
Thus, we can find $\zeta>0$ such that
$$
\leb ({S}_{n}({\Delta^c}))\leq \zeta\sigma^{n/2}.
$$
This obviously implies that the part of the sum respecting $\Delta^c$ in~\eqref{eq.sumsum} is finite.

Consider now $k\ge n_0$ and $n\ge k$.
Fix  $U \in \mathcal{U}_k$ and consider ${S}_{n}(U)$ the
$n$-satellite associated to it.  By definition of $S_{n}(U)$ and
 Lemma \ref{estpreball} we have
\begin{equation}\label{item1}
  \leb(S_{n}(U)) \leq \sum_{m=0}^{N_0}\leb \left( \bigcup _{x\in
    H_{n}^m(U)} V_{n}(x) \right) 
  \leq C_3\sum_{m=0}^{N_0}\leb \left(\bigcup_{x\in
    H_{n}^m(U)} U_{n,m}^x\right)
\end{equation}
For the last step observe that for fixed $n,m$ and given $x,x'\in H_n^m(U)$ one must have
$U^{x}_{n,m}=U^{x'}_{n,m}$ or $U^{x}_{n,m}\cap U^{x'}_{n,m}=\emptyset$.
%
%
%
Thus,  by  Lemma
\ref{estimativas}
\begin{eqnarray}
  \leb(S_{n}(U))
  &\leq& C_3C_5(N_0+1)\sigma^{\frac{n-k}{2}}\leb(U)\nonumber
  \end{eqnarray}
Letting 
$C=C_3C_5(N_0+1)$ it follows that
\begin{eqnarray*}
 \sum_{k=n_0}^\infty\sum_{U\in
\mathcal{U}_k}\sum_{n=k}^\infty\leb(S_n(U))
&\le&C\sum_{k=n_0}^\infty\sum_{U\in
\mathcal{U}_k}\sum_{j=0}^{\infty}\sigma^{j/2}\leb(U)\\
&=& C\frac{1}{1-\sigma^{1/2}}\sum_{k=n_0}^\infty\sum_{U\in
\mathcal{U}_k} \leb(U)\\
&\leq & C\frac{1}{1-\sigma^{1/2}}\leb(\Delta).
\end{eqnarray*}
This gives the conclusion.
\cqd


\subsection{Integrability of the inducing times}
\label{s.inducing}

In the previous sections we proved the existence of a Lebesgue mod 0
partition $\mathcal{P}$ of~$\Delta$ and an inducing time function
\(R: \Delta \to \mathbb N \) which is constant in the elements of
$\cp$. Moreover, the map $F:\Delta\to \Delta$ defined for
$F(x)=f^{R(x)}(x)$ is a $C^2$ piecewise uniformly expanding
 map with
uniform bounded distortion. By \cite[Lemma~
 2]{Y1} such a map has a
unique absolutely continuous (with respect to Lebesgue measure)
ergodic invariant probability measure $\nu$ whose density is bounded
away from zero and infinity by
constants.  Thus in particular the integrability with respect to Lebesgue
of the return time function \( R \)
is equivalent to the integrability with respect to \( \nu \). To complete the proof it is therefore sufficient to prove the following

\begin{Proposition}
The inducing time function $R$ is $\nu$-integrable.
\end{Proposition}
\begin{proof}
We first introduce some notation.
For \( x\in \Delta \) we consider the orbit \( x, f(x),..., f^{n-1}(x) \) of the point \( x \) under iteration by \( f \) for some large value of \( n \). In particular \( x \) may undergo several full returns to \( \Delta \) before time \( n \).  Then we define the following quantities:
\begin{align*}
 H^{(n)}(x) &:= \text{number of hyperbolic times for \( x \) before time \( n \)}
 \\
  S^{(n)}(x) &:=\text{number of times \( x \) belongs to a satellite before time \( n \)}
  \\
   R^{(n)}(x)&:= \text{number of returns of \( x \) before time \( n \)}
\end{align*}
Each time that \( x \) has a hyperbolic time, it either then  has a return within some finite and uniformly bounded number of iterations, or by definition it belongs to a satellite. Therefore there exists some constant \( \kappa>0 \) independent of \( x \) and \( n \) such that
\[   R^{(n)}(x)+S^{(n)}(x) \geq \kappa H^{(n)}(x)
 \]
Notice that \( x \) may belong to a satellite or have a return without it having a hyperbolic time itself, since it may belong to a hyperbolic pre-ball of some other point \( y \) which has a hyperbolic time.
Dividing the above equation through by \( n \)  we get
\[
\frac{R^{(n)}(x)}{n}+\frac{S^{(n)}(x)}{n} \geq \frac{\kappa H^{(n)}(x)}{n}
\]
 Recalling that hyperbolic times have uniformly positive asymptotic frequency, there exists a constant \( \theta>0 \) such that \( H^{(n)}(x)/n \geq \theta \) for all \( n \) sufficiently large, and therefore, rearranging the left hand side above gives
 \[
 \frac{R^{(n)}(x)}{n}\left(1+\frac{S^{(n)}(x)}{R^{(n)}(x)}\right) \geq  \kappa\theta > 0
  \]
 Moreover \( S^{(n)}(x)/R^{(n)}(x) \) converges by Birkhoff's ergodic theorem to precisely the average number of times \( \int S d\nu \) that typical points belong to satellites before they return, and from Proposition~\ref{prop.Sn} it follows that \( \int S d\nu < \infty \). Therefore,
 we have
 \begin{equation}
 \label{key}   \frac{R^{(n)}(x)}{n} \geq \kappa' > 0
 \end{equation}
 for all sufficiently large \( n \) where \( \kappa'\) can be chosen arbitrarily close to \(\kappa\theta/(1+\int S d\nu) \) which is independent of \( x \) and $n$.
 To conclude the proof notice that  \( n/R^{(n)}(x) \) is precisely the average return time over the first \( n \) iterations and thus converges by Birkhoff's ergodic theorem to \( \int R d\nu \). This holds even if we do not assume a priori that \( R \) is integrable since it is a positive function and thus \( \int R d\nu \) is always well defined and lack of integrability necessarily implies \( \int R d\nu = + \infty \).
 Thus. arguing by contradiction and assuming that   \( \int Rd\nu = + \infty  \) gives
\( {n}/{R^{(n)}(x)} \to  \int Rd\nu = + \infty \)
and therefore
\( {R^{(n)}(x)}/{n} \to 0\).  This contradicts \eqref{key} and therefore implies that we must have \( \int R d\nu < +\infty \)
 as required.
\end{proof}


\end{document}